\documentclass[11pt,a4paper,reqno]{amsart}
\usepackage{microtype}
\usepackage[utf8]{inputenc}

\usepackage{amsmath}
\usepackage{amsthm}
\usepackage{amsfonts}
\usepackage{pgf,tikz}
\usepackage{tikz-cd}
\usepackage{amssymb}
\newcommand{\Ss}{\textsection}

\usepackage{xcolor}
\usepackage{ulem}

\newcommand\Zz{\mathbb Z}

\newcommand\Cc{\mathbb C}
\newcommand\Pp{\mathbb P}

\newcommand{\CC}{\mathcal{C}}
\newcommand{\Cd}{\mathcal{C}_d}

\DeclareMathOperator{\rk}{\text{rk}}
\DeclareMathOperator{\dive}{\text{div}}
\DeclareMathOperator{\iso}{\cong}

\DeclareMathOperator{\Db}{\text{D}^b}
\DeclareMathOperator{\NS}{\text{NS}}
\DeclareMathOperator{\Pic}{\text{Pic}}

\DeclareMathOperator{\Stab}{\text{Stab}}

\newcommand{\prim}{\text{prim}}

\newcommand{\pp}{\cdot}

\newtheorem{mainthm}{Theorem}
\newtheorem{theorem}{Theorem}
\numberwithin{theorem}{section}

\newtheorem{corollary}[theorem]{Corollary}

\newtheorem{proposition}[theorem]{Proposition}
\newtheorem{example}[theorem]{Example}
\newtheorem{question}[theorem]{Question}

\theoremstyle{remark}
\newtheorem{remark}[theorem]{Remark}

\usepackage{color}

\usepackage[colorlinks=true,linkcolor=blue,citecolor=blue
]{hyperref}%

\usepackage{booktabs}
\usepackage{graphicx}
\usepackage{colortbl}
\usepackage{float}

\title[On HK manifolds of K3$^{[n]}$-type with large Picard number]{On Hyperk\"ahler manifolds of K3$^{[n]}$-type with large Picard number}

\author[Y. Prieto--Monta\~{n}ez]{Yulieth Prieto--Monta\~{n}ez} %
\address{Pontificia Universidad Católica de Chile, Campus San Joaquín, Avenida Vicuña Mackenna 4860, Santiago de Chile, Chile} %
\email{yulieth.prieto@uc.cl}

\date{}
\begin{document}
\maketitle

\vspace{-2em}

\begin{abstract}
Inspired by well-known examples of hyperkähler manifolds, we show that any hyperkähler manifold $X$ of K3$^{[n]}$-type with Picard number $\rho(X) \geq 4$ is always isomorphic to a moduli space of twisted stable sheaves on a K3 surface. Additionally, we provide explicit descriptions of hyperkähler manifolds of K3$^{[n]}$-type with Picard ranks below this crucial value (e.g., $\rho(X)=3$) that are not birational to such moduli spaces.
\end{abstract}


\section{Introduction}

Kähler manifolds whose first Chern class is trivial are distinguished varieties known for their rich and intricate geometry. Beauville-Bogomolov's Theorem in \cite{Beau83} reveals a remarkable decomposition of such manifolds (after a finite étale cover), involving complex tori, strict Calabi--Yau manifolds, and hyperkähler manifolds (HK). Furthermore, in \cite{Beau83}, Beauville also introduces two major families of compact hyperkähler manifolds: the Hilbert scheme of \( n \) points on K3 surfaces and the generalised Kummer variety. In \cite{BeaDon85}, Beauville and Donagi introduced an example arising from parametrizing lines in cubic fourfolds with the same type of deformation as the Hilbert scheme of two points on a K3 surface. Under specific conditions on the cubic fourfold, they also proved that its variety of lines not only deforms but can be isomorphic to the Hilbert scheme of two points on a K3 surface of degree 14. The notion of deformation type became crucial in understanding hyperkähler manifolds, especially for determining uniqueness among known examples or the possibility of discovering new hyperkähler manifolds. We will restrict our attention to projective hyperkähler manifolds of K3$^{[n]}$-type which are those deforming to the Hilbert scheme of \(n\) points on a K3 surface.\\

Seminal works due to Mukai on moduli spaces of bundles on K3 surfaces introduce additional examples of Hyperkähler manifolds of K3\(^{[n]}\)-type. Restricting the choice of sheaves on the K3 surface is important for ensuring good stability conditions and a well-behaved moduli space \cite{Gie77}. In fact, these restrictions translate into numerical conditions (e.g., specific Chern classes) on the sheaves to guarantee properties such as smoothness and the natural symplectic structure of the moduli space, as demonstrated in \cite{Muk84, Muk87}. These generalizations expanded the list of Hyperkähler manifolds in dimension at least 2 and established the Hilbert scheme of $n$ points on a K3 surface as the special case of the moduli space of stable rank one sheaves on a K3 surface with the second Chern number \( n \). Thanks to the contributions of Huybrechts and Göttsche, O'Grady, and Yoshioka \cite{GotHuy96, Ogr97, Yos01}, explicit constructions of moduli spaces of stable sheaves on K3 surfaces have expanded Mukai’s examples and introduced additional hyperkähler manifolds, particularly those deforming to the Hilbert scheme of \( n \) points on K3 surfaces. Similarly, Yoshioka in \cite{Yos06} proved that some moduli spaces of twisted stable sheaves on K3 surfaces are also examples of compact Hyperkähler manifolds of K3$^{[n]}$-type. Without a doubt, moduli spaces are not the only examples of this type of variety (see, for instance, \cite{LehLehSorVan17}); however, they abound in this particular family. One way to understand this phenomenon is by studying the period map and, therefore, the cohomology of this family as the most faithful topological invariant. Classical results in lattice theory provide useful characterizations for studying the specific case of the moduli space of sheaves on K3 surfaces in the family of hyperkähler manifolds of K3$^{[n]}$-type when the Picard lattice becomes large.

Let $X$ be a hyperk\"ahler manifold of K3$^{[n]}$-type. The second cohomology group $H^2(X,\Zz)$, endowed with the Beauville--Bogomolov--Fujiki form, is a lattice of discriminant $-2n+2$, signature $(3,20)$ and so, $$H^2(X,\Zz) \cong U^3\oplus E_8(-1)^2\oplus \langle -2n+2 \rangle.$$ Consider the even unimodular lattice $\widetilde\Lambda$ of signature $(4,20)$, isomorphic to $U^4\oplus E_8(-1)^2$, commonly referred to as the (abstract) Mukai lattice. This is because the Mukai lattice $\widetilde{H}(S,\Zz)$ of any K3 surface $S$ is isomorphic to $\widetilde{\Lambda}$. In \cite[Corollary 9.5]{Mar11} (cf. \cite[\S1]{BayHassTsc15} or \cite[Theorem 3]{Add16}), Markman showed a natural primitive embedding of lattices preserving the Hodge structures:
\begin{equation}\label{eq: Markman emb}
    H^2(X,\Zz) \subset \widetilde \Lambda.
\end{equation} 

Assuming the embedding in \eqref{eq: Markman emb}, the orthogonal complement of $H^2(X,\Zz)$ in $\widetilde\Lambda$ is generated by a primitive vector $v_X \in \widetilde\Lambda$ with $v_X^2=2n-2 \geq 0$. Specifically, if \(X=M_H(S,v)\), a moduli space of $H$-stable sheaves on a K3 surface \(S\) with a primitive Mukai vector \(v \in \widetilde{H}(S, \mathbb{Z})\) and $H$ is an ample class $v$-generic, we know that the extension \(H^2(X, \mathbb{Z}) \subset \widetilde{\Lambda}\) is defined as follows: if \(v^2 = 0\), then \(v^\perp / v \subset \widetilde{H}(S, \mathbb{Z})\); if \(v^2 > 0\), then \(v^\perp \subset \widetilde{H}(S, \mathbb{Z})\) \cite{Muk84, Ogr97}.

Denote by $\widetilde \Lambda^{1,1}$ the algebraic part of $\widetilde\Lambda$ consisting of integral classes of type $(1,1)$ and containing $H^{1,1}(X,\Zz)$ with respect to the embedding in (\ref{eq: Markman emb}). In other words, since $\NS(X)=H^2(X,\Zz)\cap H^{1,1}(X)$ for any hyperk\"ahler manifold, we obtain $\NS(X) \subset \widetilde \Lambda^{1,1}$. Additionally, the class $v_X$ is of type $(1,1)$ in $\widetilde\Lambda$. In \cite[Proposition 4]{Add16}, Addington identified hyperk\"ahler manifolds of K3$^{[n]}$-type that are birational to moduli spaces of stable sheaves on K3 surfaces by a primitive embedding of a hyperbolic plane $U$ in $\widetilde\Lambda^{1,1}$. Later, in \cite[Lemma 2.6]{Huy17}, Huybrechts generalized this criterion for moduli space of twisted stable sheaves on K3 surfaces by requiring an embedding of a (not necessary primitive) twisted hyperbolic plane $U(k) \subset \widetilde \Lambda^{1,1}$ for some $k \in \Zz$. Applications of such criteria are useful in studying family of hyperkähler manifolds admitting birational maps of finite order. For instance, if the action of such maps preserves the symplectic 2-form, then by some results in lattice theory, there exists an embedding $U(k)\subset \widetilde{\Lambda}^{1,1}$ for some $k \in \Zz$, and so, hyperk\"ahler manifolds admitting a non-trivial symplectic birational map are always birational to some moduli space of twisted stable sheaves on a K3 surface, \cite{Pri22} cf. \cite[Main Theorem A]{DutDomPri24}.\\

It has been established that the birational geometry of a moduli space \( M_H(S,v) \), can be analyzed using wall-crossing techniques based on Bridgeland stability conditions. Denote by $\Stab(S)$ the space of stability conditions on the derived category of $S$. Assume that $v$ is a primitive vector and $v^2\geq 0$. For any generic stability condition \( \sigma \), the moduli space of \( \sigma \)-stable objects of class \( v \), denoted by \( M_{\sigma}(S,v) \), is a smooth projective hyperkähler manifold deformation-equivalent to the Hilbert scheme of points on a K3 surface (see \cite[Theorem 2.15]{BayMac14}). Moreover,  there exists a choice of a chamber \( \mathcal{C} \subset \text{Stab}(S) \) where the Gieseker stability condition in \( M_H(S,v) \) is recovered, that is, \( M_\mathcal{\sigma}(v) \iso M_H(S,v) \) for all $\sigma \in \mathcal{C}$.

This approach leads to the conclusion:

\begin{mainthm}\label{main: Theorem A}
    Let $X$ be a projective hyperk\"ahler manifold of K3$^{[n]}$-type, $n\geq 2$, of Picard number at least four. Then, $X$ is isomorphic to some moduli space of twisted stable sheaves on a K3 surface.
\end{mainthm}

\begin{proof}
    Due to the widely recognized Hasse-Minkowski theorem, any indefinite lattice $L$ with a rank of at least five represents zero \cite[IV. \S 3.2 Corollary 2]{Ser73}. In particular, assuming $\rho(X)$ is at least four and considering the class $v_X$ of $(1,1)$-type given by the embedding in \eqref{eq: Markman emb}, these conditions guarantee the existence of a non-trivial element $D \in \widetilde{\Lambda}^{1,1}$, which is an isotropic class (i.e., $D^2 = 0$). Let $k = \dive(D)$ denote the divisibility of $D$. Now, choose a class $F \in \widetilde\Lambda^{1,1}$ such that $\langle F \cdot D\rangle=k$ and $F^2=2m$ for some $m$. Then, it is straightforward to show $\langle D, -mD +kF\rangle$ generates $U(k^2) \subset \widetilde\Lambda^{1,1}$. As we mentioned above, this is enough to conclude that $X$ is birational to some moduli space $M_H(S, \alpha, v)$ of $\alpha$-twisted stable sheaves on some K3 surface $S$ with Mukai vector $v$ due to \cite{Add16, MonWan15, Huy17}. A birational map between two hyperkähler manifolds may or may not extend to an isomorphism. If it does, then $X \cong M_H(S, \alpha, v)$.
    
    If this is not the case, then $X \cong M_{H'}(S', \beta, w)$, a moduli space of twisted $H'$-stable sheaves on a K3 surface $S'$, constructed as follows: by \cite{Bri08} (later generalized to the twisted case in \cite{HuyMacSte08}), any $M_H(S, \alpha, v)$ can be recovered as a moduli space $M_\sigma(S, v)$ of $\sigma$-stable objects for some $\sigma \in \Stab^+(S)$, where $\Stab^+(S)$ denotes a connected component of $\Stab(S)$. According to \cite[Theorem 1.1, 1.2]{BayMac14}, any birational model of $M_\sigma(S, v)$ (and so apply to $X$) appears as a moduli space $M_{\sigma'}(S, v)$, where $\sigma' \in \Stab^+(S)$. Without loss of generality, we say $X \cong M_\sigma(S, v)$ for some $\sigma \in \Stab^+(S)$ a $v$-generic stability condition.
   
  Consider an isotropic Mukai vector $v_0 \in \widetilde{H}(S, \mathbb{Z})$. Since $\sigma \in \Stab^+(S)$ is a generic stability condition with respect to the Mukai vector $v$, by \cite[Lemma 7.2]{BayMac14-projectivity}, there exists a class $\beta \in \operatorname{Br}(M_\sigma(S, v_0))$ and a derived equivalence:
   $$\phi: \Db(S) \rightarrow \Db(M_\sigma(S, v_0), \beta).$$
   In particular, $S' := M_\sigma(S, v_0)$ is a smooth projective K3 surface with a Brauer class $\beta$. Again, by \cite{BayMac14-projectivity}, the derived equivalence $\phi$ induces an isomorphism $\widetilde{\phi}: M_{\sigma}(S, v) \cong M_{H'}(S', \beta, \phi(v))$ where $H'$ is an ample class on $S'$. The stability condition in the latter moduli space is now Gieseker but with respect to a K3 surface $S'$ derived equivalent to $S$. Therefore, $X$ is isomorphic to the moduli space of twisted $H'$-stable sheaves on a K3 surface $S'$ with Mukai vector $w:=\phi(v)$ and $H'$ is an ample class $w$-generic.
    
\end{proof}

Even though the existence proof appears quite simple, it is not trivial to identify the K3 surface and, moreover, the numerical conditions of the sheaves encoded in the Mukai vector. As a consequence of a result by Morrison \cite[Corollary 2.10]{Mor84}, if \( S \) denotes the K3 surface where \( X \) is isomorphic to \( M_H(S,\alpha, v) \), and assuming that \( \rho(X) \geq 13 \), we obtain $\NS(S)$ is uniquely determined by its rank $\rho(S) \geq 12$ and its discriminant form. In this particular case, the class $\alpha$ is the trivial ones, and so, $X$ is isomorphic to some moduli space of (non-twisted)-stable sheaves. Moreover, there exists an embedding \( U \subset \NS(S) \), so, $S$ also admits an elliptic fibration (with section).\\

On the other hand, there exist moduli spaces of (twisted) sheaves on K3 surfaces of K3$^{[n]}$-type with a Picard number less than 4. Based on these considerations, the following question naturally arises:

\begin{question}\label{main question}
  Is 4 the smallest integer such that any projective hyperkähler manifold of K3$^{[n]}$-type with Picard number \(\rho(X) \geq 4\) is always isomorphic to some moduli space of twisted stable sheaves on a K3 surface?
\end{question}

There exist general hyperk\"ahler manifolds of K3$^{[2]}$-type that are not isomorphic to moduli spaces of sheaves on K3 surfaces. Notable examples have been extensively discussed in the works by Hassett \cite{Has00} considering the variety of lines on cubic fourfolds. Additional examples of intricate geometry are also explored in \cite{Has00}, where the geometry of a cubic fourfold \(Y\) is explicitly described via certain divisors \(\mathcal{C}_d\) in the moduli space of cubic fourfolds. This geometric structure is mirrored in the algebraic cohomology of \(Y\), with the presence of algebraic surfaces in \(Y\) further shaping the geometry of the hyperkähler fourfold \(F(Y)\). This is the approach that we adopt to produce examples of hyperk\"ahler manifolds different from moduli spaces of twisted stable sheaves on K3 surfaces.

In \cite{MacSte12}, it is shown that when the cubic fourfold \(Y\) contains a plane (i.e., $Y \in \CC_8$), the Fano variety \(F(Y)\) is not necessarily isomorphic to a moduli space of \(H\)-stable sheaves on a K3 surface. Instead, it parameterizes twisted stable complexes of sheaves on a K3 surface. However, assuming \(Y\)  has Picard number two, an isomorphism to a non-twisted version can be established. Lately, in \cite{Huy17}, Huybrechts extended this result to more values of \(d\). The minimum value of \(d\) for which \(F(Y)\) is not birational to a moduli space of twisted \(H\)-stable sheaves on a K3 surface is $d=12$, indicating that the cubic fourfold \(Y\) contains a cubic scroll. The next example occurs when \(d=20\), which means that \(Y\) contains a Veronese surface. For the special generic case (meaning that $\rho(F(Y))=2$), it is sufficient to consider $Y \in \CC_d$, where $d$ is an integer such that $d>6$, $d \equiv 0,2 \pmod 6$ and the prime factorization of $d/2=\Pi p_i^{n_i}$ satisfies the following condition: 
\begin{equation}\label{eq: not be a moduli twisted}
    \exists p_i \equiv 2 \pmod 3 \text{ such that } n_i \text{ is odd.}\tag{$\dagger$}
\end{equation}
   
However, for $\rho(F(Y))\geq 3$, the intersection of two Hassett divisors $\CC_{d_1}$ and $\CC_{d_2}$, can be nesting and induce labels $K_{d}$ where the condition \eqref{eq: not be a moduli twisted} fails even if $d_1$ and $d_2$ satisfies such condition. In fact, we have the following theorem:

\begin{mainthm}\label{main: Theorem example rk 3 bir to moduli}
    Let $Y$ be a cubic fourfold with a rational normal cubic scroll $\Sigma_3$ and a Veronese surface $V$ in $Y$. Then, the variety of lines $F(Y)$ is isomorphic to a moduli space of twisted stable sheaves on some K3 surface.\\
\end{mainthm}

A crucial point in ensuring that \( 4 \) is the “optimal” bound involves showing that there exists a hyperk\"ahler variety with Picard number \( 3 \) that cannot be birational to a moduli space of twisted stable sheaves on a K3 surface but still is deformation equivalent to the Hilbert scheme of finite length on a K3 surface. To achieve this, we study the non-emptiness of the intersection of Hassett divisors in order to get a desired \( F(Y) \). We start with the case when \(\rk(F(Y))=3\), as analyzed by Yang--Yu in \cite{YanYu20}. We prove:

\begin{mainthm}\label{main: theorem rk 3 non bir to moduli}
    Let $Y$ be a cubic fourfold admitting a Veronese surface. Suppose $Y \in \CC_{60}$ and the rank of the algebraic cohomology of $Y$ is three. Then, the variety of lines $F(Y)$ is a hyperk\"ahler fourfold with $\rho(F(Y))=3$ deforming to the Hilbert scheme of two points on a K3 surface but non-birational to any moduli space of twisted stable sheaves on a K3 surface.
\end{mainthm}
A consequence of Theorems  \ref{main: Theorem A} and \ref{main: theorem rk 3 non bir to moduli} is the following fact:

\begin{corollary}\label{main: non bir to moduli has rk<4}
    If \(X\) is a projective hyperkähler manifold of K3$^{[n]}$-type non-birational to any moduli space of twisted stable sheaves on a K3 surface, then \(1 \leq \rho(X) \leq 3\).
\end{corollary}

For the case where \(\text{rk}(F(Y)) = 4\), we begin by generalizing some results from \cite{YanYu20}. We analyze the intersection of the divisors \(\CC_{d_1}\), \(\CC_{d_2}\), and \(\CC_{d_3}\), where each \(d_i\) satisfies condition \eqref{eq: not be a moduli twisted}. Nevertheless, due to Theorem \ref{main: Theorem A}, an inductive construction like the one above cannot be extended to cases where $\rho(F(Y)) \geq 4\). We analyze the particular case of cubic fourfolds with $\rk A(Y)= 4$. The following result is a consequence of classical facts on number theory concerning the solvability of polynomial equations over finite fields:

\begin{mainthm}\label{main: theorem rk>3}
    If \(Y \in \CC_{d_1}\cap \CC_{d_2}\cap \CC_{d_3}\) is a special cubic fourfold with \(\rk A(Y) \geq 4\), then $Y \in C_d$ for some $d$ that does not satisfy condition\eqref{eq: not be a moduli twisted}. In particular, \(F(Y)\) is isomorphic to some moduli space of twisted stable sheaves on a K3 surface.
\end{mainthm}

The paper is structured as follows: In Section \ref{section: main defns of cubic}, we review key facts about cubic fourfolds, their varieties of lines, and the intersection of Hassett divisors. In Section \ref{section: example in rk 3 with F(Y) non birational to moduli}, we present explicit hyperkähler manifolds of rank three derived from special cubic fourfolds, as discussed in Theorem \ref{main: Theorem example rk 3 bir to moduli}. We also outline conditions under which these cubic fourfolds yield the hyperkähler fourfold described in Theorem \ref{main: theorem rk 3 non bir to moduli}. Finally, in Section \ref{section: rk >=4}, we find some applications of Theorem \ref{main: Theorem A} and an alternative proof for the notable case of the variety of lines of cubic fourfolds, thereby proving Theorem \ref{main: theorem rk>3}.

\subsection*{Acknowledgments}

It is a great pleasure to thank Justin Sawon, Emanuele Macrì, and Laura Pertusi for their impetus in starting this project, as well as for their insightful discussions and valuable suggestions on the moduli of cubic fourfolds and the intersection of Hassett divisors. I would also like to thank Kieran G. O'Grady who showed interest in the main question and encouraged me to write the manuscript during his visit to ICTP. I am grateful to Lothar Göttsche and Laura Pertusi for reading a preliminary version of this paper and for their observations. Thanks also to Daria Tieplova, Mohamed Aliouane, and Giacomo Mezzedimi for their valuable comments. I am deeply grateful to Reinder Meinsma for kindly highlighting that the proof of Theorem \ref{main: Theorem A} applies when $\rho(X)$ is at least four.

\subsection*{Notation}
 For a projective hyperkähler manifold $X$ of K3$^{[n]}$-type, $n \geq 2$ and its Picard number is denoted by $\rho(X)$, which is defined as $\rho(X) = \text{rk}\, \Pic(X) = \text{rk}\, \NS(X) = \text{rk}\, H^{1,1}(X,\mathbb{Z})$. Lattices are free abelian groups equipped with a non-degenerate symmetric bilinear form. The ADE lattices considered here are positive definite.

\section{Cubic fourfolds and their varieties of lines}\label{section: main defns of cubic}

We refer the reader to \cite{Has00,Has16} for a deeper discussion of cubic fourfolds. We review some basic notions and relevant material pertinent to our discussion.

Let $Y \subset \Pp^5$ be a smooth cubic fourfold defined by the zero locus of a homogeneous polynomial $P \in \Cc[x_0,x_1,x_2,x_3,x_4,x_5]$ of degree three. Denote by $h$ the restriction of the hyperplane class (it is also known as the Poincar\'e dual to the hyperplane class), so $h^4=\deg(Y)=3$. Hence, by Lefschetz hyperplane theorem and Poincaré duality, $H^2(Y,\Zz)=\Zz \pp h$ and $H^6(Y,\Zz)=\Zz \pp (h^3/3)$. The Picard group of $Y$ is $\Pic(Y) \iso \Zz \pp h$ and the Betti numbers of $Y$ are given by 
\[
b_i(Y) := \dim H^i(Y,\Zz)=
\begin{cases}
  1 & \text{if } i=0,2,6,8, \\
  23& \text{if } i =4, \\
  0& \text{otherwise.}
\end{cases}
\]

In particular, the middle Hodge numbers of $Y$ are
$$h^{4,0}(Y)=h^{0,4}(Y)=0, \ h^{3,1}(Y)=h^{1,3}(Y)=1, \text{ and } h^{2,2}(Y)=21.$$ The algebraic cohomology of $Y$ is denoted by $A(Y):=H^4(Y,\Zz)\cap H^{2,2}(Y)$ and it is a positive definite lattice containing $h^2$. 

It was proved by Zucker in \cite{Zuc77} that \( A(Y) \) is generated (over \( \mathbb{Q} \)) by the classes of algebraic cycles, thereby confirming the validity of the Hodge conjecture for cubic fourfolds. Lately, Voisin in \cite[Theorem 1.4]{Voi13} proved the integral version of the Hodge conjecture for cubic fourfolds.

\begin{remark}\label{rmk: very general, special, very generic and generic cubic}
A cubic fourfold $Y$ is called very general if any algebraic surface is homologous to a complete intersection. In particular, its algebraic lattice $A(Y)$ has rank $1$, and so $A(Y)$ is generated by the class $h^2$. Following the notation introduced by Hassett in \cite[\S 3]{Has00}, $Y$ is special if it contains an algebraic surface $T \subset Y$ not homologous to a complete intersection. Specifically, this condition is equivalent to requiring that the rank of $H^{2,2}(Y,\mathbb{Z})$ is \textit{at least two}, see \cite[\S 3]{Has00} and if it is exactly two, we say that $Y$ is special generic. As a consequence of the validity of the Integral Hodge conjecture for cubic fourfolds, any class $\gamma \in H^{2,2}(Y,\Zz)$ is algebraic, see \cite{Voi13}. Since \( T \) can be \( \mathbb{P}^2 \), the cubic \( Y \) is namely very generic if it admits a plane \( \mathbb{P}^2 \subset \mathbb{P}^5 \), and the rank of the algebraic part of \( H^2(Y, \mathbb{Z}) \) is exactly two.
\end{remark}

Denote by $F(Y)$ the variety of lines $\mathbb{P}^1 \simeq L \subset Y$ contained in a cubic fourfold $Y$. It is well-known that $F(Y)$ is an irreducible smooth projective fourfold with vanishing first Chern class. Beauville and Donagi, in \cite{BeaDon85}, showed that it is indeed a hyperk\"ahler manifold of dimension $4$ and is deformation-equivalent to the Hilbert scheme of two points of a K3 surface. Particularly, the Betti numbers of $F(Y)$ are given by 
\[
b_i(F(Y)) := \dim H^i(F(Y),\Zz)=
\begin{cases}
  1 & \text{if } i=0\\
  23& \text{if } i =2, \\
  276& \text{if } i =4, \\
  0& \text{otherwise.}
\end{cases}
\]

and the non-trivial Hodge numbers of $F(Y)$ are:

$$h^{2,0}=h^{0,2}=h^{4,0}=h^{0,4}=1, \ h^{1,1}=21, \ h^{3,1}=h^{1,3}=21 \text{ and } h^{2,2}=232.$$

Note that $\rho(F(Y)):=\rk(\NS(F(Y))=\rk (H^{1,1}(F(Y))\cap H^2(F(Y),\Zz))$ is an integral number between $1$ and $21$. In particular, for a very general smooth cubic fourfold $Y$, $\rho(F(Y)) = 1$. Some examples, including those studied by Beauville and Donagi, showed that under certain assumptions on $Y$ (for instance, the case of Pfaffian cubic fourfolds), the Fano variety $F(Y)$ can be isomorphic to the Hilbert scheme of $2$-points on a general K3 surface (of degree 14). This observation implies that the Picard number of $F(Y)$ can be also $2$, as the exceptional divisor introduces an additional class. \\

The cohomology structure of $H^4(Y,\Zz)$ under the intersection form $\langle , \rangle$ is a unique (up to isometries) odd unimodular lattice of signature $(21,2)$, isomorphic to $$(H^4(Y,\Zz),\langle , \rangle) \cong (+1)^{\oplus 21}\oplus (-1)^{\oplus 2}.$$ Additionally, according to \cite[Proposition 2.1.2]{Has00}, the primitive cohomology $H^4(Y,\Zz)_{\prim}$ is identified as $$H^4(Y,\Zz)_{\prim}=(h^2)^\perp \cong A_2 \oplus U^{\oplus 2} \oplus E_8^{\oplus 2}$$ due to the isomorphism $H^4(Y,\Zz)_{\prim} \overset{\sim}{\rightarrow} H^2(F(Y),\Zz)_{\prim}(-1)$ induced by the Abel-Jacobi map $\alpha: H^4(Y,\Zz)\rightarrow H^2(F(Y),\Zz)(-1)$; refer to \cite[Proposition 6]{BeaDon85} for details. \\

\subsection{Hassett divisors and the moduli of cubic fourfolds}
We review some key aspects of the moduli space of special cubic fourfolds in \cite[\Ss 3]{Has00}. We focus on numerical criteria that determine when the variety of lines is birational to a moduli space of (twisted) stable sheaves on a K3 surface.\\

Denote by $\Cd$ the Hassett (Noether-Lefschetz) divisor in the moduli space of special cubic fourfolds of discriminant $d$. A special cubic fourfold $Y$ belongs to $\CC_d$ if there is a primitive sublattice $K \subset A(Y)$ of rank 2 and discriminant $d$ containing the class $h^2$. The integer $d$ represents the determinant of the intersection form on $K$. If $S\subset Y$ is an algebraic surface (not homologous to a complete intersection surface), then:
$$S^2:=\langle S,S\rangle=c_2(N_{S/Y})=6(h|S)^2 + 3(h|S \cdot K_S) + K_S^2-\chi(S),$$ where $\chi(S)$ is the topological Euler characteristic of the surface $S$. 

\begin{remark} Assume that \(\text{rk}\,A(Y) = 2\) and fix a basis $h^2, S$ of \(A(Y)\). Then a basis for the rank-2 lattice \(K\) with discriminant \(d\) can be given by $h^2, \alpha = xh^2 + yS$, where \(x, y \in \mathbb{Z}\). Now, by computing the determinant of the bilinear form in the basis of \(K\), we obtain \(d = Dy^2\), where \(D\) is the determinant of the bilinear form in \(A(Y)\). By \cite[Proposition 3.2.4]{Has00}, we find that \(d\) and \(D\) coincide. Then, the fact that \(K \subseteq A(Y)\) is primitive and preserves \(h^2\) implies that \(y = \pm1\). Therefore, the basis of \(K\) is $h^2, xh^2 \pm S$ for any \(x \in \mathbb{Z}\). A special cubic fourfold is called typical if it has unique labelling.
\end{remark}

We recall some numerical characterizations of the divisors \(\CC_d\) that establish the variety of lines on cubic fourfolds is birational to specific examples of hyperkähler manifolds of K3$^{[n]}$-type. 

\begin{theorem}[\cite{Has00}, \cite{Add16}, \cite{Huy17}]\label{Theorem: moduli of cubic fourfolds}
Suppose that $Y \in \Cd$.
\begin{itemize}
    \item The divisor $\Cd$ is irreducible and nonempty if and only if 
    \begin{equation}\label{cond be non empty}
        d > 6 \text{ and } d \equiv 0,2 \pmod{6}. \tag{$\ast$} 
    \end{equation}
\end{itemize}
Furthermore, if $Y \in \Cd$ with $d$ satisfies the condition (\ref{cond be non empty}), then:
\begin{itemize}
    \item $F(Y)$ is birational to a Hilbert scheme of two points for some K3 surface if and only if $d$ is of the form
    \begin{equation}\label{cond be bir to Hilbert sch}
         d=\frac{2k^2+2k+2}{a^2} \text{ for some } k, a \in \Zz. \tag{$\ast\ast\ast$} 
    \end{equation}
    
    \item $F(Y)$ is birational to a moduli space of stable sheaves on a K3 surface if and only if 
    \begin{equation}\label{cond be bir to moduli}
        d \text{ is not divisible by } 4,9 \text{ or any odd prime } p\equiv 2 \pmod{3}. \tag{$\ast\ast$} 
    \end{equation}
    
    \item  $F(Y)$ is birational to a moduli space of twisted stable sheaves on a K3 surface if $ d/2=\Pi p_i^{n_i}$ satisfies
    \begin{equation}\label{cond be bir to moduli twisted}
         n_i \equiv 0 \pmod 2 \text{ for all } p_i\equiv 2 \pmod 3 \tag{$\ast\ast'$}    
    \end{equation}
\end{itemize}
\end{theorem}
\begin{proof}
    Condition \eqref{cond be non empty} was established in \cite[Theorem 1.0.1]{Has00}; condition \eqref{cond be bir to Hilbert sch} was demonstrated in \cite[Proposition 6.1.3, Theorem 6.1.4]{Has00} and \cite[Theorem 2]{Add16}; condition \eqref{cond be bir to moduli} was proven in \cite[Theorem 1]{Add16}; and condition \eqref{cond be bir to moduli twisted} was shown in \cite[Proposition 4.1]{Huy17}.
\end{proof}

We trivially remark that:

\begin{center}
    (\ref{cond be bir to Hilbert sch}) $\Rightarrow$ (\ref{cond be bir to moduli}) $\Rightarrow$ (\ref{cond be bir to moduli twisted}) $\Rightarrow$ (\ref{cond be non empty}). 
\end{center}

\begin{remark}
    For primes \( p_i \equiv 0, 1 \pmod{3} \) in \eqref{cond be bir to moduli twisted}, the exponent \( n_i \) can be even or odd. For example, if \( d = 14 \), then \( d/2 = 7^1 \) (or if \( d = 18 \), then \( d/2 = 3^2 \) ) is allowed. The failure of condition \eqref{cond be bir to moduli twisted} is equivalent to stating that \(d\) satisfies condition \eqref{cond be non empty} and that there exists at least one prime \(p_i \equiv 2 \pmod 3\) in its prime factorization with an odd exponent. This is exactly the condition \eqref{eq: not be a moduli twisted} mentioned in the introduction.
\end{remark}

\begin{remark}
    One of the most studied divisors in the moduli of cubic fourfolds is $\CC_8$. Note that $Y \in \mathcal{C}_8$ satisfies condition \eqref{cond be non empty} but not \eqref{cond be bir to moduli}. Nevertheless, Macrì and Stellari showed that the Fano variety $F(Y)$ for such cubic fourfolds is isomorphic to a moduli space of twisted stable complexes on a K3 surface. In fact, $d=8$ satisfies condition \eqref{cond be bir to moduli twisted} implying that $F(Y)$ is birational to a moduli space of twisted stable sheaves on a K3 surface. Furthermore, under the assumption that $Y$ is \textit{very generic}, this birational map is either an isomorphism or a Mukai flop (in the dual plane $P^\vee \subset F(Y)$); see \cite[Theorem 1.4]{MacSte12}. Cubic fourfolds in that divisor correspond to those containing a plane, typically containing other surfaces such as quadric surfaces and del Pezzo surfaces. Additionally, it was noted in \cite[Remark 4.3.ii]{MacSte12} that $F(Y)$ is birational to a moduli space of non-twisted stables sheaves of a K3 surface if additionally assume $\rk \NS(F(Y))\geq 13$. 
\end{remark}

Cubic fourfolds with $\rk A(Y) \geq 3$ can admit more than one label. Indeed, Hassett divisors can intersect, revealing the richness of the geometry of cubic fourfolds. It was noted by Fano in \cite{Fan43} that a cubic fourfold containing two disjoint planes possesses a marking with discriminant 14, meaning that $Y \in \mathcal{C}_8 \cap \mathcal{C}_{14}$, cf. \cite[\Ss 4]{Has00}.

\begin{proposition}{\cite[Theorem 7]{YanYu20}}\label{theorem: inter of Hassett div} Any two Hassett divisors  $\CC_{d_1}$, $\CC_{d_2}$ satisfying condition (\ref{cond be non empty}) intersects, i.e., $\CC_{d_1}\cap \CC_{d_2} \neq \emptyset$. Moreover, there exists a smooth cubic fourfold $Y \in \CC_{d_1}\cap \CC_{d_2}$ with a rank 3 lattice $A(Y)$ of discriminant 
\begin{equation*}
    |A(Y)| = \begin{cases} 
      \frac{d_1d_2-1}{3}, & \text{if } d_1,d_2\equiv 2 \pmod 6  \\
      \frac{d_1d_2}{3} & \text{otherwise.} 
   \end{cases}
\end{equation*}  
\end{proposition}

\begin{proposition}\label{Prop: intersection two hassett in rk 3} Let $Y$ be a cubic fourfold in $\CC_{d_1}\cap\CC_{d_2}$ where $d_1,d_2$ satisfy condition \eqref{cond be non empty}. Assume that $\rk A(Y)=3$. Then, $Y \in \CC_{d}$ where $d=d_1y^2 + 2\lambda_{d_1d_2}yz+d_2z^2$ satisfies condition \eqref{cond be non empty}, $y, z \in \Zz$ with $\gcd(y,z)=1$, and
    $$\lambda_{d_1d_2}=
    \begin{cases} 
    \pm 1 & \text{if } d_1, d_2 \equiv 2 \pmod 6 \\
     0 & \text{otherwise.} 
    \end{cases}
    $$ 
\end{proposition}
\begin{proof}
    Set by $A(Y)=\langle h^2, S,T \rangle$ a basis of the algebraic cohomology of $Y$ with intersection bilinear form
    
    \begin{table}[h!]
    \begin{tabular}{c|ccc}
          & $h^2$ & $S$ & $T$ \\ \hline
    $h^2$ & 3     & $a$ & $b$ \\
    $S$   & $a$   & s   & $c$ \\
    $T$   & $b$   & $c$ & $t$
    \end{tabular}
    \end{table}
    
    where $d_1=3s-a^2$ and $d_2=3t-b^2$ satisfy condition \eqref{cond be non empty} for two labels $K_{d_1}$ and $K_{d_2}$. Define $$\alpha:=xh^2+yS+zT \in A(Y).$$
    It is straightforward to show that 
    \begin{align}
    \left\{
    \begin{array}{ll}
    \alpha^2 & = 3x^2 +2axy+y^2s+2xzb+2zyc+z^2t, \\
    \langle h^2,\alpha \rangle & = 3x+ya+zb.
    \end{array}
    \right.
    \end{align}
    In terms of the base $h^2, \alpha$, we obtain $Y \in \CC_d$ where \begin{eqnarray}\label{eq: d in new base h alpha}
        d &=& y^2(3s-a^2)+2yz(3c-ab)+z^2(3t-b^2) \nonumber \\
        &=& d_1 y^2 + 2yz(3c-ab) + d_2z^2.
    \end{eqnarray}
    By Proposition \ref{theorem: inter of Hassett div}, the discriminant of $A(Y)$ satisfies
    \begin{eqnarray*}
    3|A(Y)|&=& 9st+6abc-3b^2s-3a^2t-9c^2\\ 
    &=& \begin{cases} 
    (3s -a^2)(3t-b^2)-1 & \text{if } d_1,d_2 \equiv 2 \pmod 6, \\
    (3s -a^2)(3t-b^2) & \text{otherwise.} 
    \end{cases}
    \end{eqnarray*}

    This imposes the condition, 
    $$\lambda_{d_1d_2}:=c-ab=
    \begin{cases} 
    \pm 1 & \text{if } d_1, d_2 \equiv 2 \pmod 6 \\
     0 & \text{otherwise.} 
    \end{cases}
    $$ 
    Replacing this in \eqref{eq: d in new base h alpha}, we obtain the possibilities of $d$. Assuming $\gcd(y,z)=1$ we obtain $\langle h^2,\alpha \rangle =K_d \subset A(Y)$ is primitive.
\end{proof}

\begin{corollary}
    Suppose that $\rk A(Y)=3$. If $Y$ admits two disjoint planes, then $Y \in \CC_8 \cap \CC_{14}$. In particular, $Y$ contains a quartic scroll or a quintic del Pezzo surface.
\end{corollary}
\begin{proof}
Since $Y \in \mathcal{C}_8$ by assuming at least one plane, Proposition \ref{Prop: intersection two hassett in rk 3} implies that $Y \in \mathcal{C}_d$, where $d = 8y^2 \pm 2yz + 8z^2$ for any $y, z \in \mathbb{Z}$, and in particular, for $y = z = \pm 1$. In fact, $\Sigma_4 = 2h^2 - P_1 - P_2$ generates an algebraic class in $A(Y)$ corresponding to a quartic scroll, and $T = h^2 + P_1 + P_2$ corresponds to a quintic del Pezzo surface.

\end{proof}

\section{A hyperk\"ahler manifold of Picard number 3}\label{section: example in rk 3 with F(Y) non birational to moduli}

To find a Hyperkähler manifold of Picard rank three that is not isomorphic to a moduli space of stable or twisted stable sheaves on a K3 surface, we consider the variety of lines on a cubic fourfold \( Y \) intersecting two specific Hassett divisors, \(\CC_{d_1}\) and \(\CC_{d_2}\), as ensured by Proposition \ref{theorem: inter of Hassett div}. However, we will see that assuming \( Y \in \CC_{d_1} \cap \CC_{d_2} \) for \( d_1, d_2 \) that satisfy condition \eqref{cond be non empty} but not \eqref{cond be bir to moduli twisted} is not always sufficient to ensure non-birationality to a moduli space of twisted stable sheaves on a K3 surface when the rank of \( A(Y) \) is three. This is the situation that is always true for the special generic case.

\begin{proposition}
    Let $Y$ be a cubic fourfold with a rational normal cubic scroll $\Sigma_3$ and a Veronese surface $V$ in $Y$. Then, 
    \begin{enumerate}
        \item $Y$ does not contain a $\Pp^2$;
        \item $Y$ does not contain a quartic scroll;
        \item the variety of lines $F(Y)$ is isomorphic to a moduli space of twisted stable sheaves on some K3 surface.
    \end{enumerate}
\end{proposition}

\begin{proof}
Denote by $A(Y)=\langle h^2, \Sigma_3,V \rangle \subset H^4(Y,\Zz)$ where $h$ corresponds to the restriction of the hyperplane class. The Gram matrix of $A(Y)$ with respect to such basis is given by \cite[\S4]{Has00} as:

    \begin{table}[h!]
    \begin{tabular}{c|ccc}
          & $h^2$ & $\Sigma_3$ & $V$ \\ \hline
    $h^2$ & 3     & 3 & 4 \\
    $\Sigma_3$   & 3   & 7   & 4 \\
    $V$   & 4   & 4 & 12
    \end{tabular}
    \end{table}

Note that $Y\not\in \CC_8$ and also $Y\not \in \CC_{14}$ but $Y \in \CC_{32}$ where the condition (\ref{cond be bir to moduli twisted}) is satisfied. Indeed, since the rank of $A(Y)$ is 3, and assuming that $Y \in \CC_8$ should ensure a class \(\alpha = xh^2 + y \Sigma_3 + zV\) with the right intersections of a label $K_8$. By computing \(\alpha^2\) and \(\langle h^2,\alpha \rangle\), we obtain:
\[
\begin{aligned}
\alpha^2 &= 3x^2 + 6xy + 7y^2 +8xz +8yz +12z^2, \\
\langle h^2,\alpha \rangle &= 3x + 3y +4z.
\end{aligned}
\]

Now, assuming that the 2-rank lattice $K_8= \langle h^2, \alpha \rangle$ has a discriminant $d=4(3y^2+5z^2)$, it is straightforward to show that there are no integers \(y, z\) satisfying the equation \(3y^2+5z^2=2\). In particular, $Y \not\in \CC_8$. Analogously, there are no integer solutions for $6y^2+10z^2=7$, and so $Y \not\in \CC_{14}$. But trivially $y=z=1$ is a solution of $3y^2+5z^2=8$, implying that $Y \in \CC_{32}$. In particular, $F(Y)$ is birational to some moduli space of sheaves on a K3 surface by Theorem \ref{Theorem: moduli of cubic fourfolds}, and the isomorphism can be obtained as the proof of Theorem \ref{main: Theorem A}. 
\end{proof}

\begin{example}
    Suppose that $d_1=20=2^2\times5$, $d_2=44=2^2\times11$, and $d_3=64=2^6$, so in particularly, $d_i/2$, $i=1,2,3$, do not satisfy condition \eqref{cond be bir to moduli twisted}. Trivially, $y=1=z$ is a solution of $$20y^2+44z^2=64.$$
    In this case, assuming $Y \in \CC_{20} \cap \CC_{44}$, the class $\alpha=h^2+S+T$ where $s:=S^2=20$ and $t:=T^2=44$, ensures that $Y \in \CC_{64}$.

   However, if $d=416$, so in particular, $d/2=208=2^4\times 13$,  satisfies condition \eqref{cond be bir to moduli twisted}. Trivially, $y=1, z=3$ is a solution of $$20y^2+44z^2=416.$$
    In this case, assuming $Y \in \CC_{20} \cap \CC_{44}$, the class $\alpha=h^2+S+3T$ where $s:=S^2=20$ and $t:=T^2=44$, ensures that $Y \in \CC_{416}$.
\end{example}

\begin{question}\label{Q: rk 3}
   For which values of \(d_1\) and \(d_2\) satisfying condition \eqref{cond be non empty} but not condition \eqref{cond be bir to moduli twisted} do we obtain all \(d\) as described in Proposition \ref{Prop: intersection two hassett in rk 3} that also do not satisfy condition \eqref{cond be bir to moduli twisted}?
\end{question}

For our purposes, it is enough to find a good pair $d_1,d_2$ to ensure that all induced $d$ are always with the desired property.

\begin{theorem}
       Let $Y$ be a cubic fourfold with a Veronese surface. Assume that $Y \in \CC_{60}$. Then, $F(Y)$ is a hyperk\"ahler manifold of dimension 4 deforming to the Hilbert scheme of two points on a K3 surface which is not birational to a moduli space of twisted stable sheaves on a K3 surface. 
       
       Moreover, $\rho (F(Y)) =3$ and if $A(Y)$ is generated by $h^2,V$ and a surface $T$, then $T$ satisfies the following numerical restrictions: $c_2(N_{T/Y})=T^2=3k^2+20$, $\langle V\cdot T\rangle=4$ and $\langle h^2\cdot T\rangle=3k$ for some $k \in \Zz_{>0}$. 
\end{theorem}

\begin{proof}
Assuming that \(Y\) contains a Veronese surface is equivalent to \(Y \in \CC_{20}\). Applying Proposition \ref{Prop: intersection two hassett in rk 3} for \(d_1 = 20 = 2^2 \times 5\) and \(d_2 = 60 = 2^2 \times 3 \times 5\), we obtain \(\lambda_{20,60} = 0\). In particular, \(Y \in \CC_d\) for all \(d\): $$ d = 20y^2 + 60z^2 = 2^2 \cdot 5(y^2 + 3z^2) \equiv 0,2 \pmod 6,$$ where $\gcd(y,z)=1$

It is straightforward to show that \(y^2 + 3z^2 = 5\) has no integer solutions, and this holds for any power of 5. This is enough to conclude that for any \(y, z \in \mathbb{Z}_{>0}\), where \(d\) satisfies \(d \equiv 0, 2 \pmod{6}\), $5$ is a prime number in the factorization of \(d/2\) with an odd exponent. It follows from Theorem \ref{Theorem: moduli of cubic fourfolds} that for all \(d\) satisfying condition \eqref{cond be non empty} but not \eqref{cond be bir to moduli twisted}, hence the variety of lines $F(Y)$ is not birational to some moduli space of twisted stable sheaves on a K3 surface.

Now, assume that $\langle h^2,V,T\rangle$ is a basis of the algebraic lattice of $Y$ with bilinear form $\langle V\cdot T\rangle=c$ and $\langle h^2\cdot T\rangle=b$.

Since \( Y \in \CC_{60} \), this implies that choosing \( K_{60} = \langle h^2, T \rangle \) as a basis, then \( 60 = 3t - b^2 \) where \( b = \langle h^2 \cdot T\rangle \) and the formula from \cite[Section 4.1]{Has00} implies that $ T^2 = c_2(N_{T/Y})$.
Hence, any $b=3k$ and $T^2=3k^2+20$,  for all $k>0$ is a solution of $3t-b^2=60$. This imposes $c=\langle T \cdot V\rangle=4k$. 

\begin{table}[H]
        \begin{tabular}{c|ccc}
              & $h^2$ & $V$ & $T$ \\ \hline
        $h^2$ & $3$     & $4$ & $3k$ \\
        $V$   & $4$   & $12$   & $4k$ \\
        $T$   & $3k$   & $4k$ & $3k^2+20$
        \end{tabular}
        \end{table}
    
\end{proof}

\begin{remark}
Similarly, we can prove that \( Y \in \mathcal{C}_{60} \cap \mathcal{C}_{180} \) ensures that \( F(Y) \) is never birational to some moduli space of twisted stable sheaves on a K3 surface since \( d/2 = 2 \times 5 \times 3(y^2 + 3z^2) \) has 5 as a prime number with an odd exponent. However, it is not enough if \( Y \in \mathcal{C}_{20} \cap \mathcal{C}_{180} \) because there exists a \( d = 2^2 \times 5(y^2 + 3^2z^2) \) where 5 has an even power. In other words, for \( y = 1 \) and \( z = 1 \), we obtain \( d = 4 \times 5 \times 5 \times 2 \), so in particular, if \( Y \in \mathcal{C}_{20} \cap \mathcal{C}_{60} \cap \mathcal{C}_{180} \), then \( Y \in \mathcal{C}_{200} \), which trivially satisfies condition \eqref{cond be bir to moduli twisted}.
\end{remark}

\section{The case of the Picard number at least 4}\label{section: rk >=4}

In this section, we prove Theorem \ref{main: theorem rk>3}, which confirms the veracity of Theorem \ref{main: Theorem A} for specific hyperk\"ahler fourfolds.

Similar to the case with Picard number three, we aim to identify appropriate \(d_1\), \(d_2\), and \(d_3\) such that for \(Y \in \CC_{d_1} \cap \CC_{d_2} \cap \CC_{d_3}\), the variety of lines \(F(Y)\) is a hyperkähler fourfold that is not birational to any moduli space of twisted stable sheaves on a K3 surface. However, we demonstrate that the inductive process fails due to numerical constraints.

As a result of the non-empty intersection of two Hassett divisors, we can derive formulas similar to those in Proposition \ref{Prop: intersection two hassett in rk 3} for the possible determinants of the bilinear matrix of \( A(Y) \) when \(\rk A(Y) = 4\).

\begin{proposition}\label{prop: inter of 3 Hassett div}
Let $Y$ be a cubic fourfold with $\rk A(Y)=4$. Suppose that $Y\in \CC_{d_1}\cap \CC_{d_2}\cap \CC_{d_3}$ where $d_1,d_2,d_3$ satisfy condition \eqref{cond be non empty}. 

Then, the discriminant $D$ of the bilinear form of $A(Y)$ is given as follows
\begin{enumerate}
    \item \label{rk 4 c1}if $d_1,d_2 \equiv 0 \pmod 6$ and $d_3\equiv 0,2 \pmod 6$, then $$D=\frac{d_1d_2d_3}{9};$$
    \item \label{rk 4 c2} if $d_1\equiv 0 \pmod 6$ and $d_2,d_3 \equiv 2 \pmod 6$, then $$D=\frac{d_1d_2d_3-d_1}{9};$$
    \item \label{rk 4 c3} if $d_1,d_2,d_3 \equiv 2 \pmod 6$, then $$D=\frac{d_1d_2d_3-d_1-d_2-d_3-2}{9}.$$
\end{enumerate}
\end{proposition}

\begin{proof} Set a basis of $A(Y)=\langle h^2,S,T,U\rangle$ and denote by $A_Y$ its intersection matrix.
    Our proof is similar to the case of rank three in \cite[Theorem 7]{YanYu20}. By definition, an integer $d$ satisfies condition \eqref{cond be non empty} and so $d>6$ and $d\equiv 0, 2 \pmod 6$. The proof is also divided into four cases:

    \textbf{Case 1:} Suppose that $d_1,d_2,d_3 \equiv 0 \pmod 6$. Set by $d_i=6n_i$, $i=1,2,3$. Then, the Gram matrix $A_Y$ with respect to the basis $A(Y)$ is 
    $$A = \begin{pmatrix}
    3 & 0 & 0 & 0 \\
   0 & 2n_1 & 0 & 0 \\
    0 & 0 & 2n_2 & 0 \\
    0 & 0 & 0 & 2n_3
    \end{pmatrix}.
    $$ Then, $|A_Y|=3(2n_1)(2n_2)(2n_3)=\frac{d_1d_2d_3}{9}$.

    \textbf{Case 2:} Suppose that $d_1,d_2 \equiv 0 \pmod 6$ and $d_3 \equiv 2 \pmod 6$. Set by $d_i=6n_i$, for $i=1,2$ and $d_3=6n_3 +2$. Then, the Gram matrix $A_Y$ with respect to the basis $A(Y)$ is 
    $$A = \begin{pmatrix}
    3 & 0 & 0 & 1 \\
   0 & 2n_1 & 0 & 0 \\
    0 & 0 & 2n_2 & 0 \\
    1 & 0 & 0 & 2n_3+1
    \end{pmatrix}.
    $$ Hence, $|A_Y|=8n_1n_2+24n_1n_2n_3=\frac{d_1d_2d_3}{9}$.

    \textbf{Case 3:} Suppose that $d_1 \equiv 0 \pmod 6$ and $,d_2, d_3 \equiv 2 \pmod 6$. Set by $d_1=6n_1$ and $d_i=6n_i+2$ for $i=2,3$. Then, the Gram matrix $A_Y$ with respect to the basis $A(Y)$ is 
    $$A = \begin{pmatrix}
    3 & 0 & 1 & 1 \\
   0 & 2n_1 & 0 & 0 \\
    1 & 0 & 2n_2+1 & 0 \\
    1 & 0 & 0 & 2n_3+1
    \end{pmatrix}.
    $$ Then, $|A_Y|=2n_1+8(n_1n_2+n_1n_3)+24n_1n_2n_3=\frac{d_1d_2d_3-d_1}{9}$.

    \textbf{Case 4:} Suppose that $d_i \equiv 2 \pmod 6$ for all $i=1,2,3$. Set by $d_i=6n_i+2$ for $i=1,2,3$. Then, the Gram matrix $A_Y$ with respect to the basis $A(Y)$ is 
    $$A = \begin{pmatrix}
    3 & 1 & 1 & 1 \\
   1 & 2n_1+1 & 0 & 0 \\
    1 & 0 & 2n_2+1 & 0 \\
    1 & 0 & 0 & 2n_3+1
    \end{pmatrix}.
    $$ Then, $|A_Y|=2(n_1+n_2+n_3)+8(n_1n_2+n_1n_3+n_2n_3)+24n_1n_2n_3$, and replacing $n_i=(d_i-2)/6$, we obtain 
    \begin{eqnarray*}
        D&=&\frac{1}{9}[ 3(d_1 + d_2 + d_3 - 6) + \\
        & &2 ( (d_1 - 2)(d_2 - 2) + (d_1 - 2)(d_3 - 2) + (d_2 - 2)(d_3- 2) ) + \\
        & &(d_1 - 2)(d_2 - 2)(d_3 - 2). ]
    \end{eqnarray*}
    Therefore, the simplified form matches the expression in \eqref{rk 4 c3}.
\end{proof}

\begin{remark}
A single value called the algebraicity index for a cubic fourfold $Y$ was introduced by Laza in \cite{Laz21} as $\kappa_Y = \frac{2^{\text{rk}(A(Y))}}{|A(Y)|}$. Roughly speaking, this invariant measures the irrationality of $Y$, and he conjectures that $Y$ is irrational when $\kappa_Y \leq 1$. It would be interesting to analyze the previous cases when $9D > 144$ and test for the existence of potentially irrational cubic fourfolds.
\end{remark}

Let $Y \subset \Pp^5$ be a cubic fourfold. Set by $A(Y)=\langle h^2, S, T, U \rangle$ with intersection matrix given as
\begin{table}[h!]
\begin{tabular}{c|cccl}
                         & $h^2$ & $S$                     & $T$                     & $U$ \\ \hline
$h^2$                    & $3$   & $a$                     & $b$                     & $e$ \\
$S$                      & $a$   & $s$                     & $c$                     & $f$ \\
$T$                      & $b$   & $c$                     & $t$                     & $g$ \\
\multicolumn{1}{l|}{$U$} & $e$   & \multicolumn{1}{l}{$f$} & \multicolumn{1}{l}{$g$} & $u$
\end{tabular}
\end{table}

Denote by $$d_1=3s-a^2,\ d_2=3t-b^2 \text{ and } d_3=3u-e^2.$$
\begin{proposition}\label{Prop: intersection three hassett in rk 4}
    Let $Y$ be a cubic fourfold in $\CC_{d_1}\cap\CC_{d_2}\cap\CC_{d_3}$ where $d_1,d_2,d_3$ satisfy condition \eqref{cond be non empty}. Assume that $\rk A(Y)$ is four. Then, $Y \in \CC_{d}$ where $$d:=d_1y^2+d_2z^2+d_3w^2+2(\lambda_{d_1d_2}yz+\lambda_{d_1d_3}yw+\lambda_{d_2d_3}zw),$$ and $y,z,w \in \Zz$ satisfies $\gcd (x,y,z)=1$. 
\end{proposition}

\begin{proof}
    Set by $A(Y)=\langle h^2, S,T,U \rangle$ a basis of the algebraic cohomology of $Y$. Define $$\alpha:=xh^2+yS+zT+wU \in A(Y).$$
    It is straightforward to show that 
    \begin{align}
    \left\{
    \begin{array}{ll}
    \alpha^2 & = 3x^2 +sy^2+tz^2+uw^2+2axy+2bxz+ 2exw+2cyz  +2fyw+2gzw, \\
    \langle h^2,\alpha \rangle & = 3x+ay+bz+ew.
    \end{array}
    \right.
    \end{align}
    
    In terms of the base $h^2, \alpha$, we obtain $Y \in \CC_d$ where \begin{eqnarray}\label{eq-rk 4: d in new base h alpha}
        d &=& (3s-a^2)y^2+(3t-b^2)z^2+(3u-e^2)w^2+ \nonumber \\
        & &2(3c-ab)yz+2(3f-ae)yw+2(3g-be)yz   \\
        &=& d_1 y^2+ d_2z^2+ d_3w^2 + 2(3c-ab)yz+2(3f-ae)yw+2(3g-be)yz.\nonumber
    \end{eqnarray}
    Assuming $\gcd(y,z,w)=1$ we obtain $K_d \subset A(Y)$ primitive. As in Proposition \ref{Prop: intersection two hassett in rk 3}, the non-empty intersection of two Hassett divisors reduces the computation of $3c-ab$, $3f-ae$, and $3g-be$ depending of numerical conditions of $d_1,d_2$ and $d_3$. Hence,
    $$\lambda_{d_1d_2}:=3c-ab=
    \begin{cases} 
    \pm 1 & \text{if } d_1, d_2 \equiv 2 \pmod 6 \\
     0 & \text{otherwise;} 
    \end{cases}
    $$ 

    $$\lambda_{d_1d_3}:=3f-ae=
    \begin{cases} 
    \pm 1 & \text{if } d_1, d_3 \equiv 2 \pmod 6 \\
     0 & \text{otherwise;} 
    \end{cases}
    $$
 
    $$\lambda_{d_2d_3}:=3g-be=
    \begin{cases} 
    \pm 1 & \text{if } d_2, d_3 \equiv 2 \pmod 6 \\
     0 & \text{otherwise.} 
    \end{cases}
    $$
    Replacing this in \eqref{eq-rk 4: d in new base h alpha}, we obtain the possibilities of $d$:

\begin{table}[h!]
\begin{tabular}{@{}ccccccc@{}}
$d_1$ & $d_2$ & $d_3$ & $\lambda_{d_1d_2}$ & $\lambda_{d_1d_3}$ & $\lambda_{d_2d_3}$ & $d$                                   \\ \midrule
0     & 0     & 0     & 0              & 0              & 0              & $d_1y^2+d_2z^2+d_3w^2$                \\
0     & 0     & 2     & 0              & 0              & 0              & $d_1y^2+d_2z^2+d_3w^2$                \\
0     & 2     & 2     & 0              & 0              & $\pm 1$        & $d_1y^2+d_2z^2+d_3w^2 \pm 2zw$        \\
2     & 2     & 2     & $\pm 1$        & $\pm 1$        & $\pm 1$        & $d_1y^2+d_2z^2+d_3w^2\pm 2(yz+yw+zw)$ \\ \bottomrule
\end{tabular}
\end{table}
\end{proof}

\begin{corollary}\label{cor: multiple scalar of Hassett divisors}
Assume that $\rk A(Y)$ is at least three. Let $Y \in \Cd$, $d \equiv 0,2 \pmod 6$. If $k$ is a perfect square, then $Y \in \CC_{kd}$.  
\end{corollary}

\begin{theorem}
    If $Y \in \CC_{d_1}\cap \CC_{d_2}\cap \CC_{d_3}$ is a special cubic fourfold with $\rk A(Y) \geq 4$, then $Y \in C_d$ for some $d$ satisfying condition \eqref{cond be bir to moduli twisted}. In particular, the variety of lines $F(Y)$ is isomorphic to some moduli space of twisted stable sheaves on some K3 surface.
\end{theorem}

\begin{proof}
The fact that $F(Y)$ is isomorphic to some moduli space of twisted stable sheaves on some K3 surface follows from Theorem \ref{main: Theorem A} but we can also provide an alternative proof coming from the numerical characterization in Theorem \ref{Theorem: moduli of cubic fourfolds}. To do this, we show the existence of some labelling $d$ of $Y$ satisfying always the condition \eqref{cond be bir to moduli twisted}. Without loss of generality, we assume that $\rk A(Y)=4$.

In general, the polynomial \( d \) is quadratic in exactly \( \rk(A(Y)) - 1 \) variables. According to a classical theorem by Chevalley \cite{Che35}, there always exists a non-trivial solution to a homogeneous polynomial in \( \mathbb{F}_p \) if the number of variables exceeds the degree of the polynomial.
To apply this in our case, set by $d' = \gcd(d_1, d_2, d_3)$ and consider its prime factorization of $$d'= p_1^{n_1}\ldots p_m^{n_m}=\prod_{n_i \text{ even}}p_i^{n_i}\prod_{n_i \text{ odd}}p_i^{n_i}$$ Denote by $I$ the set of all $p_i \equiv 2 \pmod{3}$ in the prime factorization of $d'$ with $n_i$ odd. Now, consider $d$ as in Proposition \ref{Prop: intersection three hassett in rk 4} and factorize each $d$ as $d = d'\cdot q_2(y, z, w)$, where $q_2$ is a quadratic form. By Chevalley's theorem, the polynomial $q_2(y, z, w)=0 \pmod {p_i}$ has a nontrivial solution since $\deg(q_2) = 2 < 3$ for any prime $p_i$. In fact, a non-trivial solution $(y_0,z_0,w_0)$ exists for all  $\prod\limits_{i \in I} p_i$:

$$q_2(y_0,z_0,w_0)=l\cdot \prod\limits_{i \in I}p_i \text{ for some }  l\in\Zz_{>0}.$$

If $l$ is an integer satisfying condition \eqref{cond be bir to moduli twisted}, then $d_0:=d'q_2(y_0,z_0,w_0)$ satisfies also condition \eqref{cond be bir to moduli twisted} by construction. Suppose that $l$ admits a prime $p \equiv 2 \pmod 3$ in its prime factorization with an odd exponent. Set by $d_0=d'\cdot q_2(\frac{y_0}{l},\frac{z_0}{l},\frac{w_0}{l})$. Then, 
\begin{eqnarray*}
    d_0 &=& \prod_{n_i \text{ even}}p_i^{n_i} \prod \limits_{i \in I}p_i^{n_i}\cdot \frac{1}{l^2} \cdot q_2(y_0,z_0,w_0)\\
    &=&\frac{1}{l}\prod_{n_i \text{ even}}p_i^{n_i}\prod \limits_{i \in I}p_i^{n_i+1}.
\end{eqnarray*}

In particular, $ Y \in \CC_{l d_0}$ if $ld_0\equiv 0,2 \pmod 6$. Note that $l d_0$ has all primes $p_i\equiv 2 \pmod 3$ with exponent even. In the case that \( ld_0 \not\equiv 0,2 \pmod{6} \), it follows from Corollary \ref{cor: multiple scalar of Hassett divisors} that \( Y \in \CC_{l^2 d_0} \). Additionally, the property that preserves all primes \( p_i \equiv 2 \pmod{3} \) with even exponents still holds for \( l^2d_0 \).
\end{proof}

\begin{remark}
The rationality of cubic fourfolds is a challenging problem that is linked to many frameworks, particularly with a significant connection to K3 surfaces, see for instance \cite[\S 3]{Has16}. Investigating the rationality of certain cubic fourfolds \(Y\), for which the variety of lines is not birational to a moduli space on a K3 surface, would be particularly insightful. Is a cubic fourfold expected to be rational if its variety of lines is a moduli space of sheaves on a K3 surface?
\end{remark}

\bibliographystyle{alpha} 
\bibliography{bibliography}

\begin{thebibliography}{DMPM24}

\bibitem[Add16]{Add16}
Nicolas Addington.
\newblock On two rationality conjectures for cubic fourfolds.
\newblock {\em Math. Res. Lett.}, 23(1):1--13, 2016.

\bibitem[BD85]{BeaDon85}
Arnaud Beauville and Ron Donagi.
\newblock La vari\'{e}t\'{e} des droites d'une hypersurface cubique de
  dimension {$4$}.
\newblock {\em C. R. Acad. Sci. Paris S\'{e}r. I Math.}, 301(14):703--706,
  1985.

\bibitem[Bea83]{Beau83}
Arnaud Beauville.
\newblock Vari\'et\'es {K}\"ahleriennes dont la premi\`ere classe de {C}hern
  est nulle.
\newblock {\em J. Differential Geom.}, 18(4):755--782, 1983.

\bibitem[BHT15]{BayHassTsc15}
Arend Bayer, Brendan Hassett, and Yuri Tschinkel.
\newblock Mori cones of holomorphic symplectic varieties of {K}3 type.
\newblock {\em Ann. Sci. \'Ec. Norm. Sup\'er. (4)}, 48(4):941--950, 2015.

\bibitem[BM14a]{BayMac14}
A.~{Bayer} and E.~{Macr\`{\i}}.
\newblock {Projectivity and birational geometry of Bridgeland moduli spaces}.
\newblock {\em {J. Am. Math. Soc.}}, 27(3):707--752, 2014.

\bibitem[BM14b]{BayMac14-projectivity}
Arend Bayer and Emanuele Macr\`i.
\newblock Projectivity and birational geometry of {B}ridgeland moduli spaces.
\newblock {\em J. Amer. Math. Soc.}, 27(3):707--752, 2014.

\bibitem[Bri08]{Bri08}
Tom Bridgeland.
\newblock Stability conditions on {$K3$} surfaces.
\newblock {\em Duke Math. J.}, 141(2):241--291, 2008.

\bibitem[Che35]{Che35}
C.~Chevalley.
\newblock D\'emonstration d'une hypoth\`ese de {M}. {A}rtin.
\newblock {\em Abh. Math. Sem. Univ. Hamburg}, 11(1):73--75, 1935.

\bibitem[DMPM24]{DutDomPri24}
Yajnaseni Dutta, Dominique Mattei, and Yulieth Prieto-Montañez.
\newblock {On Symplectic Birational Self-Maps of Projective Hyperkähler
  Manifolds of K3$^{[n]}$-Type}.
\newblock {\em International Mathematics Research Notices}, page rnae112, 05
  2024.

\bibitem[Fan43]{Fan43}
Gino Fano.
\newblock Sulle forme cubiche dello spazio a cinque dimensioni contenenti
  rigate razionali del {$4^\circ$} ordine.
\newblock {\em Comment. Math. Helv.}, 15:71--80, 1943.

\bibitem[GH96]{GotHuy96}
L.~G\"{o}ttsche and D.~Huybrechts.
\newblock Hodge numbers of moduli spaces of stable bundles on {$K3$} surfaces.
\newblock {\em Internat. J. Math.}, 7(3):359--372, 1996.

\bibitem[Gie77]{Gie77}
D.~Gieseker.
\newblock On the moduli of vector bundles on an algebraic surface.
\newblock {\em Ann. of Math. (2)}, 106(1):45--60, 1977.

\bibitem[Has00]{Has00}
Brendan Hassett.
\newblock Special cubic fourfolds.
\newblock {\em Compositio Mathematica}, 120(1):1–23, 2000.

\bibitem[Has16]{Has16}
Brendan Hassett.
\newblock {\em Cubic Fourfolds, K3 Surfaces, and Rationality Questions}, pages
  29--66.
\newblock Springer International Publishing, 2016.

\bibitem[HMS08]{HuyMacSte08}
Daniel Huybrechts, Emanuele Macr\`i, and Paolo Stellari.
\newblock Stability conditions for generic {$K3$} categories.
\newblock {\em Compos. Math.}, 144(1):134--162, 2008.

\bibitem[Huy17]{Huy17}
Daniel Huybrechts.
\newblock The {K}3 category of a cubic fourfold.
\newblock {\em Compos. Math.}, 153(3):586--620, 2017.

\bibitem[Laz21]{Laz21}
Radu Laza.
\newblock Maximally algebraic potentially irrational cubic fourfolds.
\newblock {\em Proc. Amer. Math. Soc.}, 149(8):3209--3220, 2021.

\bibitem[LLSvS17]{LehLehSorVan17}
Christian Lehn, Manfred Lehn, Christoph Sorger, and Duco van Straten.
\newblock Twisted cubics on cubic fourfolds.
\newblock {\em J. Reine Angew. Math.}, 731:87--128, 2017.

\bibitem[Mar11]{Mar11}
Eyal Markman.
\newblock A survey of {T}orelli and monodromy results for
  holomorphic-symplectic varieties.
\newblock In {\em Complex and differential geometry}, volume~8 of {\em Springer
  Proc. Math.}, pages 257--322. Springer, Heidelberg, 2011.

\bibitem[Mor84]{Mor84}
D.~R. Morrison.
\newblock On {$K3$} surfaces with large {P}icard number.
\newblock {\em Invent. Math.}, 75(1):105--121, 1984.

\bibitem[MS12]{MacSte12}
Emanuele Macr\`{i} and Paolo Stellari.
\newblock Fano varieties of cubic fourfolds containing a plane.
\newblock {\em Math. Ann.}, 354(3):1147--1176, 2012.

\bibitem[Muk84]{Muk84}
Shigeru Mukai.
\newblock Symplectic structure of the moduli space of sheaves on an abelian or
  {$K3$} surface.
\newblock {\em Invent. Math.}, 77(1):101--116, 1984.

\bibitem[Muk87]{Muk87}
S.~Mukai.
\newblock On the moduli space of bundles on {$K3$} surfaces. {I}.
\newblock In {\em Vector bundles on algebraic varieties ({B}ombay, 1984)},
  volume~11 of {\em Tata Inst. Fund. Res. Stud. Math.}, pages 341--413. Tata
  Inst. Fund. Res., Bombay, 1987.

\bibitem[MW15]{MonWan15}
G.~Mongardi and M.~Wandel.
\newblock Induced automorphisms on irreducible symplectic manifolds.
\newblock {\em J. Lond. Math. Soc. (2)}, 92(1):123--143, 2015.

\bibitem[O'G97]{Ogr97}
Kieran~G. O'Grady.
\newblock The weight-two {H}odge structure of moduli spaces of sheaves on a
  {$K3$} surface.
\newblock {\em J. Algebraic Geom.}, 6(4):599--644, 1997.

\bibitem[PM22]{Pri22}
Yulieth~Katterin Prieto-Monta{\~n}ez.
\newblock {\em Automorphisms on algebraic varieties: K3 surfaces,
  hyperk{\"a}hler manifolds, and applications on Ulrich bundles}.
\newblock PhD thesis, alma, Marzo 2022.

\bibitem[Ser73]{Ser73}
J.-P. Serre.
\newblock {\em A course in arithmetic}.
\newblock Graduate Texts in Mathematics, No. 7. Springer-Verlag, New
  York-Heidelberg, 1973.
\newblock Translated from the French.

\bibitem[Voi13]{Voi13}
Claire Voisin.
\newblock Abel-{J}acobi map, integral {H}odge classes and decomposition of the
  diagonal.
\newblock {\em J. Algebraic Geom.}, 22(1):141--174, 2013.

\bibitem[Yos01]{Yos01}
Kota Yoshioka.
\newblock Moduli spaces of stable sheaves on abelian surfaces.
\newblock {\em Math. Ann.}, 321(4):817--884, 2001.

\bibitem[Yos06]{Yos06}
Kota Yoshioka.
\newblock Moduli spaces of twisted sheaves on a projective variety.
\newblock In {\em Moduli spaces and arithmetic geometry}, volume~45 of {\em
  Adv. Stud. Pure Math.}, pages 1--30. Math. Soc. Japan, Tokyo, 2006.

\bibitem[YY20]{YanYu20}
Song Yang and Xun Yu.
\newblock Rational cubic fourfolds in {H}assett divisors.
\newblock {\em C. R. Math. Acad. Sci. Paris}, 358(2):129--137, 2020.

\bibitem[Zuc77]{Zuc77}
Steven Zucker.
\newblock The {H}odge conjecture for cubic fourfolds.
\newblock {\em Compositio Math.}, 34(2):199--209, 1977.

\end{thebibliography}
\end{document}